\documentclass[10pt]{amsart}

\usepackage{pgf,pgfarrows,pgfnodes,pgfautomata,pgfheaps,pgfshade}
\usepackage{enumerate}
\usepackage{enumitem}
\usepackage[english]{babel}
\usepackage[utf8]{inputenc}
\usepackage{dsfont,bbm,mathrsfs}
\usepackage{amsmath, amsthm, amsbsy, amssymb}
\usepackage[all]{xy}
\usepackage{tikz}
\usepackage{tikz-cd}
\usepackage{mathtools}
\usepackage{bbold}

\usepackage{setspace}
\usepackage[top=3cm,bottom=3cm,left=2.5cm,right=2.5cm]{geometry}

\newcommand{\onto}{\twoheadrightarrow}
\newcommand{\into}{\hookrightarrow}
\newcommand{\proves}{\vdash}
\renewcommand{\d}{{\downarrow}}
\renewcommand{\u}{{\uparrow}}

\renewcommand{\theta}{\vartheta}
\renewcommand{\phi}{\varphi}
\newcommand{\IPC}{\mathrm{IPC}}

\renewcommand{\v}{\overline{p}} 
\newcommand{\vp}{\overline{p},v} 
\newcommand{\F}[1]{F(#1)} 
\newcommand{\E}[1]{E(#1)}
\newcommand{\N}{\mathbb{N}} 
\newcommand{\T}{\mathbb{T}} 
\newcommand{\f}{f} 
\newcommand{\h}{h} 

\newcommand{\w}[1]{\widehat{#1}}
\renewcommand{\o}[1]{\overline{#1}}

\theoremstyle{plain}
\newtheorem{theorem}{Theorem}
\newtheorem{proposition}[theorem]{Proposition}
\newtheorem{lemma}[theorem]{Lemma}

\theoremstyle{definition}

\newtheorem{example}[theorem]{Example}
\newtheorem{remark}[theorem]{Remark}
\newtheorem*{notation}{Notation}

\newcommand{\Addresses}{{
  \bigskip
  \footnotesize

  Samuel\ J.\ v.\ Gool, \textsc{Mathematics Department, City University of New York City College, NY 10031, USA and
  ILLC, Universiteit van Amsterdam, Postbus 94242, 1090 GE Amsterdam, The Netherlands}\par\nopagebreak
  \textit{E-mail address}, Sam\ v.\ Gool: \texttt{samvangool@me.com}

  \medskip

  Luca Reggio, \textsc{IRIF, Universit\'e Paris Diderot, Sorbonne Paris Cit\'e, Case 7014, 75205 Paris Cedex 13, France and 
  Laboratoire J.\ A.\ Dieudonn\'e, Universit\'e de Nice -- Sophia Antipolis, Parc Valrose, 06108 Nice Cedex 02, France}\par\nopagebreak
  \textit{E-mail address}, Luca Reggio: \texttt{reggio@unice.fr}

  \medskip
}}

\newcommand{\Acknowledgements}{{
\bigskip
\footnotesize

{\bf Acknowledgements.}
The authors would like to thank the anonymous referee for their useful comments and suggestions, which improved the paper. 
The first-named author was supported by the European Union's Horizon 2020 research and innovation programme under the Marie Sklodowska-Curie grant \#655941. The second-named author acknowledges financial support from Sorbonne Paris Cit\'e (PhD agreement USPC IDEX -- REGGI15RDXMTSPC1GEHRKE), and from the European Research Council (ERC) under the European Union's Horizon 2020 research and innovation programme (grant \#670624).
}
}

\begin{document}
\title{An open mapping theorem for finitely copresented Esakia spaces}
\author{Samuel\ J.\ v.\ Gool and Luca Reggio}
\thanks{To appear in \emph{Topology and its Applications}}


\maketitle
\vspace{-1cm}
\begin{abstract}
We prove an open mapping theorem for the topological spaces dual to finitely presented Heyting algebras.
This yields in particular a short, self-contained semantic proof of the uniform interpolation theorem for intuitionistic propositional logic, first proved by Pitts in 1992.
Our proof is based on the methods of Ghilardi \& Zawadowski. However, our proof does not require sheaves nor games, only basic duality theory for Heyting algebras.
\end{abstract}




%
%
\vspace{5mm}

In this paper, we give a short and self-contained proof of an open mapping theorem for dual spaces of finitely presented Heyting algebras. Our proof relies only on Esakia duality for Heyting algebras and a combinatorial argument in the spirit of \cite{GZ1995}, but avoids the machinery of sheaves and games used there. This open mapping theorem in particular yields as a corollary an alternative proof of the uniform interpolation theorem for intuitionistic propositional logic ($\IPC$), first proved in \cite{Pit1992} using proof-theoretic methods.

Uniform interpolation is a strong property possessed by certain propositional logics. On the one hand, uniform interpolants give implicit definitions of second-order quantifiers in a propositional logic \cite{Pit1992}. On the other hand, uniform interpolation is tightly related to the existence of a model completion for the first-order theory of the class of algebras associated to a logic \cite{GZ2002}. While the connection between ordinary deductive interpolation for propositional logics and amalgamation properties of the associated variety of algebras has been extensively investigated (see e.g.\ \cite{MMT14}), the first systematic study of uniform interpolation from a semantic (universal algebraic) standpoint appears to be \cite{vGMT2017}, following \cite{GZ2002}.

We believe that, more than the result itself, the contribution of our paper relies in the proof method that we adopt. The topological approach is shown to be useful for studying uniform interpolation in the case of $\IPC$. This paves the way to similar investigations for other non-classical propositional logics.  

The paper is structured as follows. In Section~\ref{s:preliminaries} we briefly recall Esakia duality for Heyting algebras, along with the relevant facts that we will use. In Section~\ref{s:interpolation-and-open-mapping-theorem} we formulate an open mapping theorem and we show how Pitts' uniform interpolation theorem follows from it. We also show there that our open mapping theorem is slightly stronger than Pitts' theorem. Sections~\ref{s:step-by-step}--\ref{s:proof-of-combinatorial-lemma} contain the proof of the main theorem. In Section~\ref{s:step-by-step} we introduce an ultrametric on the dual space, which shows how the step-by-step construction of finitely generated free Heyting algebras \cite{Ghi1995} relates to the topological setting. In Section~\ref{s:construction-finite-kripke-models}, we use this ultrametric to reduce the open mapping theorem to a lemma concerning finite Kripke models. We prove this lemma in the final Section~\ref{s:proof-of-combinatorial-lemma}.

\begin{notation}
Throughout the paper, we will employ the following notations. The set of non-negative integers is $\N:=\{0,1,2,\ldots\}$, and $\#S$ denotes the cardinality of any set $S$. Moreover, if $(X,\leq)$ is a poset and $x\in X$, write $\u{x}:=\{x'\in X\mid x\leq x'\}$ and $\d{x}:=\{x'\in X\mid x'\leq x$\}; for $S\subseteq X$, $\u{S}:=\bigcup_{x\in S}{\u{x}}$ and $\d{S}:=\bigcup_{x\in S}{\d{x}}$.
\end{notation}

\section{Esakia duality for Heyting algebras}\label{s:preliminaries}
We assume familiarity with the theory of distributive lattices; for background see, e.g., \cite[Chapters II-III]{BD74}. 
Recall that a \emph{Heyting algebra} is a bounded distributive lattice $A$ in which the operation $\wedge$ has a residual $\to$, that is $a\wedge b\leq c$ iff $b\leq a\to c$ for all $a,b,c\in A$.
An example of Heyting algebra is provided by the lattice of opens of an arbitrary topological space. 
Notice that in a Heyting algebra of this form, the supremum of any subset exists (i.e., the Heyting algebra is \emph{complete}), which is not the case in all Heyting algebras.
We next recall the basics of Esakia duality, which gives a topological representation for all Heyting algebras. See, e.g., \cite{Geh2014Esak} for more details. 
An \emph{Esakia space} is a partially ordered compact space $(X,\leq)$ such that: \emph{(i)} $X$ is totally order-disconnected, that is, whenever $x\not\leq y$ are elements of $X$, there is a clopen (=closed and open) $U\subseteq X$ that is an up-set for $\leq$ and satisfies $x\in U$ but $y\notin U$; and \emph{(ii)} $\d{C}$ is clopen whenever $C$ is a clopen subset of $X$.
Given a Heyting algebra $A$, the set $X_A$ of prime filters of $A$ partially ordered by set-theoretic inclusion is an Esakia space when equipped with the Stone topology generated by the sets $\w{a}:=\{x\in X_A\mid a\in x\}$, for $a\in A$, and their complements. 
Moreover, if $h\colon A\to B$ is a homomorphism of Heyting algebras then $f := h^{-1}\colon X_B\to X_A$ is continuous, and a \emph{p-morphism}, i.e.\ $\u{f^{-1}(S)}=f^{-1}(\u{S})$ for every subset $S\subseteq X_A$.
This correspondence yields a duality, known as Esakia duality \cite{Esak1974}, between the category of Heyting algebras and their homomorphisms, and the category of Esakia spaces and continuous p-morphisms. In particular, a Heyting algebra $A$ can be recovered, up to isomorphism, from its dual Esakia space as the algebra of clopen up-sets of $X_A$, where the assignment $a\mapsto \w{a}$ is a Heyting algebra isomorphism. 

In dealing with properties of $\IPC$, a key r\^ole is played by finitely generated free Heyting algebras and their dual spaces. Let $\F{\v}$ be the Heyting algebra free on a finite set $\v$, that is, the algebra of $\IPC$-equivalence classes of propositional intuitionistic formulae in the variables $\v$, and $\E{\v}$ its dual Esakia space. A Heyting algebra is \emph{finitely presented} if it is the quotient of $\F{\v}$ under a finitely generated congruence; such congruences can in fact always be generated by a single pair of the form $(\phi,\top)$. We call an Esakia space \emph{finitely copresented} if its Heyting algebra of clopen up-sets is finitely presented. Equivalently, an Esakia space is finitely copresented if it is order-homeomorphic to a clopen up-set of $\E{\v}$ for some finite $\v$.
We recall two basic facts about such spaces in Proposition~\ref{p:properties-of-co-free-spaces}. The first item amounts to the completeness of $\IPC$ with respect to its canonical model, and the second item is the dualization of the universal property of free algebras.
\begin{proposition}\label{p:properties-of-co-free-spaces}
Let $\v=\{p_1,\ldots,p_l\}$ be any finite set of variables.
\begin{enumerate}
\item For any two formulae $\phi(\v)$ and $\psi(\v)$, $\phi \proves_\IPC \psi$ if, and only if, $\w{\phi} \subseteq \w{\psi}$ 
as subsets of $\E{\v}$.
\item If $Y$ is an Esakia space and $C_1,\ldots,C_l$ are clopen up-sets of $Y$, there exists a unique continuous p-morphism 
$\h_Y\colon Y\to \E{\v}$ satisfying $\h_Y^{-1}(\w{p_i})=C_i$ for all $i\in\{1,\ldots, l\}$.
\end{enumerate}
\end{proposition}
\begin{proof}
For item $1$, we have $\phi \proves_\IPC \psi$ iff $[\phi]\leq[\psi]$ in $\F{\v}$, which in turn is equivalent to $\w{\phi} \subseteq \w{\psi}$ because $\w{-}$ is an isomorphism of Heyting algebras. For item $2$, note that the choice of the clopen up-sets $C_1,\ldots,C_l$ gives a function from $\v$ to the algebra of clopen up-sets of $Y$. The dual map of the unique homomorphism lifting this function is $h_Y$.
\end{proof}

\section{Open maps and uniform interpolation}\label{s:interpolation-and-open-mapping-theorem}
The main aim of this paper is to prove the following theorem.
\begin{theorem}\label{thm:fin-gen}
Every continuous p-morphism between finitely copresented Esakia spaces is an open map.
\end{theorem}
We show first that Pitts' uniform interpolation theorem follows in a straight-forward manner from Theorem~\ref{thm:fin-gen} and the Craig interpolation theorem for $\IPC$ \cite{Sch62}. Throughout the paper, $\v$ will denote a finite set of variables, and $v$ a variable not in $\v$. 
\begin{theorem}[Pitts \cite{Pit1992}]
Let $\phi(\vp)$ be a propositional formula. There exist propositional formulae $\phi_R(\v)$ and $\phi_L(\v)$ such that, for any formula $\psi(\v,\o{q})$ not containing $v$,
\begin{align*}
\phi \proves_\IPC \psi &\iff \phi_R \proves_\IPC \psi,\\
\psi \proves_\IPC \phi &\iff \psi \proves_\IPC \phi_L.
\end{align*}
\end{theorem}
\begin{proof}
By the Craig interpolation theorem for $\IPC$, it suffices to prove the statement for any formula $\psi$ whose variables are contained in $\v$ (cf., e.g., \cite[Prop.~3.5]{vGMT2017}). Since $\w{\phi}\subseteq\E{\vp}$ is a clopen up-set, it follows at once from Theorem~\ref{thm:fin-gen}, and the definitions of Esakia space and p-morphism, that $\f(\w{\phi})$ and $(\d\f(\w{\phi}^{\, c}))^c$ are clopen up-sets of $\E{\v}$. Thus there exist formulae $\phi_R(\v)$ and $\phi_L(\v)$ such that $\w{\phi_R}=\f(\w{\phi})$ and $\w{\phi_L}=(\d\f(\w{\phi}^{\, c}))^c$. 
It is easy to see, using the first part of Proposition~\ref{p:properties-of-co-free-spaces}, that $\phi_R$ and $\phi_L$ satisfy the conditions in the statement.
\end{proof}
As a first step towards proving Theorem~\ref{thm:fin-gen}, we will show that the theorem follows from a special case, namely Proposition~\ref{prop:open-mapping} below.
Denote by $i$ the embedding of free Heyting algebras $\F{\v}\into\F{\vp}$ that is the identity on $\v$. Let $f\colon \E{\vp}\onto\E{\v}$ be the continuous p-morphism dual to $i$. 
\begin{proposition}\label{prop:open-mapping}
The map $\f\colon \E{\vp}\onto\E{\v}$ is open.
\end{proposition}

\begin{proof}[Proof that Proposition~\ref{prop:open-mapping} implies Theorem~\ref{thm:fin-gen}.]
Let $g\colon X_A\to X_B$ be any continuous p-morphism between Esakia spaces. If $X_A$ and $X_B$ are dual to finitely presented Heyting algebras $A$ and $B$, respectively, then (see, e.g.,\ \cite[Lemma 3.11]{vGMT2017}) there are finite presentations $j_A\colon \F{\v,\o{q}}\onto A$ and $j_B\colon \F{\v}\onto B$ such that $j_A\circ i=g^{-1}\circ j_B$, where $i\colon \F{\v}\into\F{\v,\o{q}}$ is the natural embedding.
Dually, we have the following commutative square
\[\begin{tikzcd}
\E{\v,\o{q}} \arrow[twoheadrightarrow]{r}& \E{\v} \\
X_A \arrow[hookrightarrow]{u} \arrow{r}{g} & X_B \arrow[hookrightarrow]{u}
\end{tikzcd}\]
where the top horizontal map is open by Proposition~\ref{prop:open-mapping}. Since the presentation $j_A$ is finite, the dual map identifies the Esakia space $X_A$ with a clopen up-set of $\E{\v,\o{q}}$, so that the left vertical map is open. Therefore $g\colon X_A\to X_B$ is also open.
\end{proof}

The connection between the existence of uniform interpolants and open maps can be explained in terms of adjoints. Indeed, it was already observed in \cite{Pit1992} that the uniform interpolation theorem is equivalent to the existence of both left and right adjoints for the embeddings $\F{\v}\into\F{\vp}$.
In turn, it is not difficult to see that if a map between Esakia spaces is open, then its dual Heyting algebra homomorphism has left and right adjoints. Theorem~\ref{thm:fin-gen} implies that these properties always hold for homomorphisms between finitely presented Heyting algebras. The following example shows that the two properties are distinct in general. In this sense, our open mapping theorem establishes a slightly stronger property than uniform interpolation.
\begin{example}
We give an example of a Heyting algebra homomorphism $h\colon A\to B$ such that $h$ is both left and right adjoint, but its dual map is not open. 
For any natural number $n \geq 1$, denote by $\mathbb{n} = \{1 < \dots < n\}$ the finite chain with $n$ elements and the discrete topology. Let $X=\mathbb{1}+\mathbb{2}+\cdots$, the disjoint order-topological sum of countably many finite discrete chains, and let $\alpha X=X\cup\{\infty\}$ its one-point compactification. Extend the partial order on $X$ to a partial order on $\alpha X$ by defining $x \leq \infty$ for all $x \in \alpha X$. Denote by $\alpha\N = \N \cup \{\infty\}$ the one-point compactification of a discrete countable space, partially ordered by $x \leq y$ iff $x = y$ or $y = \infty$. Then $\alpha X$ and $\alpha\N$ are both Esakia spaces. 
Define a function $f\colon \alpha X\to \alpha \N$ by $f(\infty) := \infty$, and, for any $x\in\mathbb{n} \subseteq \alpha X$, $f(x) := n$ if $x < n$ and $f(x) := \infty$ if $x = n$.
Note that $f$ is a continuous p-morphism. Let $h \colon A \to B$ be the dual Heyting algebra homomorphism. If $U\subseteq \alpha X$ is 
a clopen up-set, then $f(U)$ is a clopen up-set, and if $V\subseteq \alpha X$ is a clopen down-set then $\d f(V)$ is a clopen down-set. 
Therefore, $h$ admits left and right adjoints. However, the map $f$ is not open.
Indeed, for any $n \geq 2$, $\mathbb{n} \subseteq X$ is open, but its image $f(\mathbb{n}) = \{n, \infty\}$ is not. 
\end{example}
\begin{remark}
The viewpoint of adjoint maps establishes a link between uniform interpolation for $\IPC$ and the theory of \emph{monadic Heyting algebras}. Indeed, recall that a monadic Heyting algebra can be described as a pair $(H,H_0)$ of Heyting algebras such that $H_0$ is a subalgebra of $H$ and the inclusion $H_0\into H$ has left and right adjoints \cite[Theorem 5]{GBezh1}. The relation between adjointness of a Heyting algebra homomorphism, and openness of the dual map, was already investigated in this framework. See, e.g., \cite[p.\ 32]{GBezh2} where an example akin to the one above is provided.
\end{remark}

\section{Clopen up-sets step-by-step}\label{s:step-by-step}
The \emph{$\to$-degree} of a propositional formula $\phi$, denoted by $|\phi|$, is the maximum number of nested occurrences of the connective $\to$ in $\phi$; $\phi$ has $\to$-degree $0$ if the connective $\to$ does not occur in $\phi$.
Fix a finite set of variables $\v$. For a point $x$ in $\E{\v}$ and $n\in\N$, we write $\T_n(x)$ for the \emph{degree $n$ theory of $x$}, that is
\begin{align*}
\T_n(x):=\{\phi(\v)\mid |\phi|\leq n \ \text{and} \ \phi\in x\}.
\end{align*}
We define a quasi-order $\leq_n$ on $\E{\v}$ by setting 
\begin{align*}
x \leq_n y \stackrel{\mathrm{def}}{\iff} \T_n(x) \subseteq \T_n(y),
\end{align*}
and we standardly define an equivalence relation $\sim_n$ on $\E{\v}$ by:
\begin{align*}
x \sim_n y \stackrel{\mathrm{def}}{\iff} x \leq_n y \text{ and } y \leq_n x \iff \T_n(x)=\T_n(y).
\end{align*}
We remark that $\bigcap_{n\in\N}{\leq_n}={\leq}$, the natural order of $\E{\v}$. Moreover, for every $n\in \N$, there are only finitely many formulae of $\to$-degree at most $n$. In particular, $\sim_n$ has finite index.
\begin{remark}\label{rmk:clopen-up-sets-n-upsets}
Notice that: \emph{a subset $S\subseteq \E{\v}$ is of the form $\w{\phi}$ for some formula $\phi(\v)$ of $\to$-degree $\leq n$ if, and only if, it is an up-set with respect to $\leq_n$. Thus, $S$ is a clopen up-set if, and only if, it is an up-set with respect to $\leq_n$ for some $n\in\N$. Hence, in particular, $\sim_n$-equivalence classes are clopen.}
In this sense, the quasi-orders $\leq_n$ yield the clopen up-sets of the space $\E{\v}$ `step-by-step'.
\end{remark}
The next proposition accounts for the Ehrenfeucht-Fraiss\'e games employed in \cite{GZ1995}. In our setting, these combinatorial structures reflect the interplay between the natural order of $\E{\v}$ and the quasi-orders $\leq_n$.
\begin{proposition}\label{prop:n-plus-one-equiv}
Suppose $x,y\in\E{\v}$ and $n\in \N$. The following equivalences hold.
\begin{enumerate}
\item $x \leq_0 y$ if, and only if, for each variable $p_i\in\v$, $p_i\in x$ implies $p_i\in y$;
\item $x \leq_{n+1} y$ if, and only if, for each $y'\in\u{y}$ there exists $x' \in\u{x}$ such that $x'\sim_n y'$.
\end{enumerate}
\end{proposition}
\begin{proof}
Item 1 follows at once from the fact that every formula $\phi(\v)$ of $\to$-degree $0$ is equivalent to a finite disjunction of finite conjunctions of variables, along with the fact that $x,y$ are prime filters.

In order to prove the left-to-right implication in item 2, assume $x \leq_{n+1} y$. Since $\sim_n$ has finite index, choose a finite set $\{y_1,\ldots,y_k\}\subseteq \u{y}$ such that each $y'\in\u{y}$ is $\sim_n$-equivalent to some $y_i$. 
It suffices to prove that for each $i\in\{1,\ldots,k\}$ there is $x_i\in\u{x}$ with $x_i\sim_n y_i$. To this aim, consider the following formula, $\phi$, of $\to$-degree $\leq n+1$, defined by
\[
\phi:= \bigvee_{i=1}^k \left(\bigwedge{\T_n(y_i)}\to\bigvee{\T_n(y_i)^{c}}\right)
\]
where the complement is relative to the set of formulae of $\to$-degree at most $n$. It follows from the definitions of the logical connectives and of $\sim_n$ that, for every $z\in\E{\v}$, 
\[
\phi\notin z \iff \forall i\in\{1,\ldots,k\} \ \exists z_i\geq z \ \text{with} \ z_i\sim_n y_i.
\]
In particular, $\phi \not\in y$. Since $x \leq_{n+1} y$, also $\phi \not\in x$.
Therefore, for each $i\in\{1,\ldots,k\}$ there is $x_i\in\u{x}$ satisfying $x_i\sim_n y_i$, as was to be shown.

For the right-to-left implication, it is enough to show that $\phi\to\psi\in y$ whenever $\phi(\v),\psi(\v)$ are formulae of $\to$-degree $\leq n$ such that $\phi\to\psi\in x$. This follows easily from the definitions and the assumption.
\end{proof}

\section{Reduction to finite Kripke models}\label{s:construction-finite-kripke-models}
Fix a finite set of variables $\v$. The Esakia space $\E{\v}$ has a countable basis, and thus admits a compatible metric by Urysohn metrization theorem, and even an ultrametric (see e.g.\ \cite[7.3.F]{Engelking}). We explicitly define such an ultrametric.
Set
\begin{equation*}\label{eq:metric}
d\colon \E{\v}\times\E{\v}\to [0,1], \ \ (x,y)\mapsto 2^{-\min\{|\phi| \, \mid \, \phi \, \in \, x\bigtriangleup y\}}
\end{equation*}
where $x\bigtriangleup y$ denotes the symmetric difference of $x$ and $y$. We adopt the conventions $\min{\emptyset}=\infty$ and $2^{-\infty}=0$.
It is immediate to check that $d$ is an ultrametric on the set $\E{\v}$, i.e.\ for all $x,y,z\in \E{\v}$ the following hold: \emph{(i)} $d(x,y)=0$ if, and only if, $x=y$; \emph{(ii)} $d(x,y)=d(y,x)$; \emph{(iii)} $d(x,z)\leq \max{(d(x,y),d(y,z))}$. 

Note that, for every $x,y \in \E{\v}$ and $n \in \mathbb{N}$, $x \sim_n y$ if, and only if, $d(x,y) < 2^{-n}$. Therefore, the open ball $B(x,2^{-n})$ of radius $2^{-n}$ centered in $x$ coincides with the equivalence class $[x]_n:=\{y\in\E{\v}\mid y\sim_n x\}$, which is clopen by Remark~\ref{rmk:clopen-up-sets-n-upsets}.
\begin{lemma}\label{l:compatible-metric}
The topology of the Esakia space $\E{\v}$ is generated by the clopen balls of the ultrametric $d$.
\end{lemma}
\begin{proof}
Observe that, for any formula $\phi(\v)$, $\w{\phi}=\bigcup_{x \in \w{\phi}}[x]_{|\phi|}=\bigcup_{x\in\w{\phi}}{B(x,2^{-|\phi|})}$. Since the latter union is over finitely many clopen sets, it follows that $\w{\phi}$ is clopen in the topology induced by the ultrametric $d$.
\end{proof}
In order to prove that the map $\f\colon \E{\vp}\onto\E{\v}$ is open, it is useful to see the spaces at hand as approximated by finite posets, in the following sense. For each $k\in \N$ consider the finite set of balls
\begin{align*}
X_k:=\{B(x,2^{-k})\mid x\in\E{\vp}\}=\{[x]_k \mid x \in \E{\vp}\},
\end{align*}
partially ordered by $\leq_k$, and write $q_k\colon \E{\vp}\onto X_k$ for the natural quotient $x\mapsto [x]_k$. For every $k'\geq k$, there is a monotone surjection $\rho_{k',k}\colon X_{k'}\onto X_k$ sending $[x]_{k'}$ to $[x]_k$. Since $\f$ is non-extensive, it can be `approximated' by the monotone map $\f_k\colon X_k\to Y_k$, $[x]_k\mapsto [\f(x)]_k$, where $Y_k:=\{B(y,2^{-k})\mid y\in\E{\v}\}$.
\[\begin{tikzcd}[row sep=scriptsize, column sep=large]
\E{\vp} \arrow{rrrr}{\f} \arrow[twoheadrightarrow]{d}[swap]{q_{k'}} & & & &\E{\v} \arrow[twoheadrightarrow]{d}{} \\
X_{k'} \arrow[dashed]{rrrr} \arrow[twoheadrightarrow]{d}[swap]{\rho_{k',k}} & & & & Y_{k'} \arrow[twoheadrightarrow]{d}{} \\
X_{k} \arrow{rrrr}{\f_k} & & & & Y_{k}
\end{tikzcd}\]

To prove the open mapping theorem for the dual spaces of free finitely generated Heyting algebras (i.e., Proposition~\ref{prop:open-mapping}), it is enough to show that for every clopen ball $B=B(x,2^{-n})$ in $\E{\vp}$, $f(x)$ lies in the interior of $f(B)$. This is equivalent to finding, for every $n$, a number $R(n)$ such that $B(f(x),2^{-R(n)}) \subseteq f(B(x,2^{-n}))$ for all $x \in \E{\vp}$. Since $f(B(x,2^{-n}))$ is closed, it suffices to construct, for any $y$ with $y \sim_{R(n)} f(x)$, a sequence $(x^m)$ in $B(x,2^{-n})$ such that $f(x^m)$ converges to $y$. For the construction of such a sequence we will use Lemma~\ref{l:combinatorial-lemma}, which is a variant of the lemmas in \cite[Section 4]{GZ1995} and in \cite[Section 5]{Vis1996}.

Before stating Lemma~\ref{l:combinatorial-lemma} and showing how it completes the above argument, we introduce some notation. Recall that a \emph{Kripke model} on the finite set of variables $\v$ (a \emph{$\v$-model}, for short) is a partially ordered set $(M,\leq)$ equipped with a monotone map $c_M\colon M\to 2^{\v}$. If $M$ is a finite $\v$-model, then by the second part of Proposition~\ref{p:properties-of-co-free-spaces} there is a unique p-morphism $\h_M\colon M\to\E{\v}$ such that $\h_M^{-1}(\w{p_i})=c_M^{-1}(\u{p_i})$ for every $p_i\in\v$. In Lemma~\ref{l:combinatorial-lemma} we will construct a $(\vp)$-model $M$ which is a sub-poset of $X_n \times Y_m$, where $m \geq n$. Given any sub-poset $M$ of $X_n \times Y_m$, we have a diagram
\begin{equation}\label{eq:Mdiagram}
\begin{tikzcd}[row sep=small]
{} & M \arrow[bend right=20]{dl}[swap]{\pi_1} \arrow[bend left=20]{dr}{\pi_2} \arrow{dd}{\xi} & \\
X_n & & Y_m \\
{} & X_{m} \arrow[bend left=20]{ul}{\rho_{m,n}} \arrow[bend right=20]{ur}[swap]{\f_{m}}  & 
\end{tikzcd}
\end{equation}
where $\xi\colon M\to X_{m}$ is defined as $\xi:=q_m\circ \h_M$ and $\pi_1\colon M\to X_n$, $\pi_2\colon M\to Y_m$ are the natural projections.
\begin{lemma}\label{l:combinatorial-lemma}
Let $n\in \N$. There exists an integer $R(n)\geq n$ such that, for every $m\geq R(n)$, there is a finite $(\vp)$-model $M$ which is a sub-poset of $X_n\times Y_m$ and satisfies the following properties:
\begin{enumerate}
\item $\{([x]_n,[y]_m)\mid y\sim_{R(n)}\f(x)\}\subseteq M$;
\item $\rho_{m,n}\circ \xi=\pi_1$;
\item $\f_{m}\circ \xi=\pi_2$.
\end{enumerate}
In particular, items (2) and (3) together correspond to the commutativity of diagram \eqref{eq:Mdiagram}.
\end{lemma}
We prove Lemma~\ref{l:combinatorial-lemma} in the next section. We conclude by showing how Proposition~\ref{prop:open-mapping}, and hence Theorem~\ref{thm:fin-gen}, follow from it.
\begin{proof}[Proof of Proposition~\ref{prop:open-mapping}]
It suffices to prove that $B(f(x),2^{-R(n)})$ is contained in $f(B(x,2^{-n}))$ for every $x \in \E{\vp}$ and $n \in \mathbb{N}$. Let $y \sim_{R(n)} f(x)$. For every $m \geq R(n)$, $([x]_n,[y]_m) \in M$ by item 1 in Lemma~\ref{l:combinatorial-lemma}; we define $x^m := h_M([x]_n,[y]_m)$. By item 2 in Lemma~\ref{l:combinatorial-lemma}, $[x^m]_n = \rho_{m,n}(\xi([x]_n,[y]_m)) = [x]_n$, so $x^m \in B(x,2^{-n})$. By item 3 in Lemma~\ref{l:combinatorial-lemma} we have $[f(x^m)]_m = f_m(\xi([x]_n,[y]_m)) = [y]_m$, so that $f(x^m)$ converges to $y$.
\end{proof}
%

\section{Proof of Lemma~\ref{l:combinatorial-lemma}}\label{s:proof-of-combinatorial-lemma}
Fix $n\in \N$. For every $x\in\E{\vp}$, define $r(x)$ 
to be the number of $\sim_n$-equivalence classes in $\E{\vp}$ above $x$, i.e.,
\[ 
r(x):= \#\{[x']_n \mid x'\in\u{x}\}=\#q_n(\u{x}).
\]
Moreover, set $R := R(n) = 2(\#X_n) - 1$.

Fix an arbitrary integer $m\geq R$. For elements $(x,y)$ and $(x',y')$ in $\E{\vp}\times \E{\v}$, we say that $(x',y')$ \emph{is a witness for} $(x,y)$ if $x' \geq x$, $y' \leq y$, $x' \sim_n x$, $\f(x) \sim_{2r(x)-1} y'$, and $\f(x') \sim_{2r(x)-2} y$. Note that, by definition, $f(x) \sim_{2r(x)-1} y$ if, and only if, $(x,y)$ is a witness for itself.

Let $M:=\{([x]_n,[y]_m) \in X_n\times Y_m \mid \text{there exists a witness for} \ (x,y)\}$,
and equip it with the product order. Defining $c_M\colon M\to 2^{(\vp)}$ by $c_M([x]_n,[y]_m):=\{u\in(\vp)\mid x\in\w{u} \, \}$ turns $M$ into a $(\vp)$-model. We prove that it satisfies the three required properties.
\begin{enumerate}[wide, labelwidth=!, labelindent=0pt]
\item If an element $([x]_n,[y]_m) \in X_n\times Y_m$ satisfies $y\sim_R \f(x)$, then $(x,y)$ is a witness for itself because $2r(x)-1\leq 2(\#X_n)-1 = R$. Therefore $([x]_n,[y]_m) \in M$.
\item Observe that $\rho_{m,n}\circ \xi=q_n\circ \h_M$. 
Hence we must show that $h_M([x]_n,[y]_m) \sim_n x$. Assume, without loss of generality, that $(x,y)$ admits a witness. We will prove by induction on $k$ that, for any $0 \leq k \leq n$,
\begin{equation}\label{eq:ind1}\tag{$P_k$}
\forall ([x]_n,[y]_m)\in M, \ \h_M([x]_n,[y]_m) \sim_k x.
\end{equation}
For $k = 0$, \eqref{eq:ind1} is true by definition of $c_M$.
We prove \eqref{eq:ind1} holds for $k+1$ provided it holds for $k\in\{0,\ldots,n-1\}$. We will show that (a) $\h_M([x]_n,[y]_m)\leq_{k+1}x$ and (b) $x\leq_{k+1} \h_M([x]_n,[y]_m)$.
\begin{description}
\item[(a)] Consider an arbitrary $w\geq x$. In view of Proposition~\ref{prop:n-plus-one-equiv} it is enough to find $z\geq \h_M([x]_n,[y]_m)$ such that $z\sim_k w$. Let $(x',y')$ be a witness for $(x,y)$. Then $x'\sim_n x$, so that there is $x''\geq x'$ with $x''\sim_{n-1} w$, whence $x''\sim_{k} w$. Now, two cases:
\begin{description}
\item[(i)] If $x''\sim_n x$, in view of the inductive hypothesis $\h_M([x]_n,[y]_m)\sim_k x$ we have $\h_M([x]_n,[y]_m)\sim_k x'' \sim_k w$.
Thus we can set $z:=\h_M([x]_n,[y]_m)$.
\item[(ii)] Else, suppose $x''\not\sim_n x$. Since $\f(x')\sim_{2r(x)-2}y$ and $\f(x'')\geq \f(x')$, there exists $z'\geq y$ with $z' \sim_{2r(x)-3}\f(x'')$. Now, $x''\not\sim_n x$ entails $r(x'')<r(x)$, hence $2r(x'')-1\leq 2r(x)-3$, showing that $(x'',z')$ is a witness for itself. Setting $z:=\h_M([x'']_n,[z']_m)$ we see that $z\geq \h_M([x]_n,[y]_m)$ because $\h_M$ is monotone, and $z\sim_k x''\sim_k w$ by the inductive hypothesis applied to $z$.
\end{description}
\item[(b)] Given an arbitrary $z\geq \h_M([x]_n,[y]_m)$ we must exhibit $w\geq x$ such that $w\sim_k z$. Since $\h_M$ is a p-morphism, there is $([x']_n,[y']_m)\geq ([x]_n,[y]_m)$ such that $\h_M([x']_n,[y']_m)=z$. By the inductive hypothesis, $\h_M([x']_n,[y']_m)\sim_k x'$. Now, $x\leq_n x'$ implies the existence of $w\geq x$ satisfying $w\sim_{n-1}x'$, therefore $w\sim_k x' \sim_k z$. 
\end{description}
\item We first prove the following claim.

{\bf Claim.} $\pi_2\colon M\to Y_m$ is a p-morphism.
\begin{proof}[Proof of Claim]
Pick $([x]_n,[y]_m)\in M$ and $z\in\E{\v}$ with $y\leq_m z$. We need to prove that there is $w\in\E{\vp}$ such that $([w]_n,[z]_m)\in M$. Suppose, without loss of generality, that $(x,y)$ admits a witness $(x',y')$. Then $\f(x)\sim_{2r(x)-1} y'\leq y\leq_m z$ entails $\f(x)\leq_{2r(x)-1} z$ because $m\geq 2r(x)-1$.
Since $f$ is a p-morphism, there exists $x''\geq x$ such that $\f(x'')\sim_{2r(x)-2}z$. We distinguish two cases, as above:
\begin{description}
\item[(i)] If $x''\sim_n x$, set $w:=x$. Then $(x'',y')$ is a witness for $(w,z)$.
\item[(ii)] If $x''\not\sim_n x$, set $w:=x''$. It is easy to see, reasoning as in case (ii) of the proof of item $(2)$, that $(w,z)$ is a witness for itself.\qedhere
\end{description}
\end{proof}
We use the claim to prove the identity $\f_{m}\circ \xi=\pi_2$. We show by induction that, for any $0 \leq k \leq m$,
\begin{equation}\label{eq:ind2}\tag{$Q_k$}
\forall ([x]_n,[y]_m)\in M, \ \f(\h_M([x]_n,[y]_m))\sim_k y.
\end{equation}
For $k = 0$, \eqref{eq:ind2} is true because $y \sim_0 f(x)$. We prove \eqref{eq:ind2} holds for $k+1$ if it holds for $k\in\{0,\ldots,m-1\}$.
As in item $2$, we prove that (a) $\f(\h_M([x]_n,[y]_m))\leq_{k+1}y$ and (b) $y\leq_{k+1} \f(\h_M([x]_n,[y]_m))$.
\begin{description}
\item[(a)] Pick $w\geq y$. By Proposition~\ref{prop:n-plus-one-equiv} it suffices to find $z\geq \f(\h_M([x]_n,[y]_m))$ with $z\sim_k w$. Since by the Claim $\pi_2$ is a p-morphism and $w\geq_m y$, there is $([x']_n,[y']_m)\in M$ such that $([x']_n,[y']_m) \geq([x]_n,[y]_m)$ and $y'\sim_m w$. Define $z:=\f(\h_M([x']_n,[y']_m))$. Then $z\geq \f(\h_M([x]_n,[y]_m))$ because $\f$ and $\h_M$ are monotone maps, and the inductive hypothesis applied to $z$ yields $z\sim_k y' \sim_k w$.
\item[(b)] The argument is the same, mutatis mutandis, as in the previous item, and it hinges on the fact that both $\h_M$ and $\f$ are p-morphisms.\qed
\end{description}
\end{enumerate}
\section*{Concluding remarks}
In this paper we have adopted a topological approach to the study of uniform interpolation for the intuitionistic propositional calculus. In particular, we have exposed the relation between uniform interpolation and open mapping theorems in topology. These kinds of connections between logical properties and topological ones are at the heart of duality theory. A well-known example is Rasiowa and Sikorski's proof \cite{RasSik1950} of G\"odel's completeness theorem for first-order classical logic, which exploited Baire Category Theorem.

It would be interesting to investigate further how Theorem~\ref{thm:fin-gen} compares to classical open mapping theorems in functional analysis (e.g.\ for Banach spaces) and in the theory of topological groups, which typically rely on an application of Baire Category Theorem. Also, it would be important to understand if similar open mapping theorems hold for other propositional logics, and what are the underlying reasons --- from a duality-theoretic perspective --- for such theorems to hold.

\providecommand{\bysame}{\leavevmode\hbox to3em{\hrulefill}\thinspace}
\providecommand{\MR}{\relax\ifhmode\unskip\space\fi MR }
\providecommand{\MRhref}[2]{%
  \href{http://www.ams.org/mathscinet-getitem?mr=#1}{#2}
}
\providecommand{\href}[2]{#2}

\Acknowledgements

\Addresses

\end{document}